%% file: Main.tex
\documentclass[11pt,fleqn]{article}

\usepackage[top=25truemm,bottom=25truemm,left=30truemm,right=30truemm]{geometry}

\usepackage{amsmath,amssymb,graphicx,amsthm,amsfonts}
\usepackage{amscd}
\usepackage{slashed}
\usepackage{enumitem}
\usepackage[all]{xy}
\usepackage{leftidx}

\theoremstyle{definition}
\newtheorem{Def}{Definition}[section]
\newtheorem{Thm}[Def]{Theorem}
\newtheorem{Prp}[Def]{Proposition}
\newtheorem{Lem}[Def]{Lemma}
\newtheorem{Cor}[Def]{Corollary}

\newtheorem{Remark}[Def]{Remark}

\def\ip<#1>{\langle #1 \rangle}

\newcommand{\mb}[1]		{ \boldsymbol{#1} }

\newcommand{\comp}		{	\circ }
\newcommand{\eps}		{ \varepsilon }

\newcommand{\deebar} 	{ \overline{\partial} }
\newcommand{\fibp}[2]	{ \,{}_{#1}\!\!\times_{#2} }

\title{A Kobayashi-Hitchin correspondence between Dirac-type singular mini-holomorphic bundles and HE-monopoles}
\author{Masaki Yoshino}
\date{\today}

\begin{document}
	\maketitle
	\begin{abstract}
		We prove an analogue of the Kobayashi-Hitchin correspondence on compact connected $3$-folds that is fibered on orbifold Riemann surfaces and satisfy an integrability condition,
		which contains compact connected Sasakian $3$-folds.
		We define mini-holomorphic bundles on such $3$-folds and the algebraic Dirac-type singularities on mini-holomorphic bundles,
		and prove that there exists a special Hermitian metric (admissible BHE-metric) on a Dirac-type singular mini-holomorphic bundle if the bundle satisfies a slope stability.
	\end{abstract}
	\input{Introduction.tex}

	\input{Sec2.tex}
	\input{Sec3.tex}

\end{document}

%% file: Introduction.tex
\section{Introduction}
	On a connected compact K\"{a}hler manifold $(M,g)$ with the K\"{a}hler form $\omega$,
	a holomorphic vector bundle $V$ has a Hermite-Einstein metric if and only if $V$ is polystable,
	which is called the Kobayashi-Hitchin correspondence and proved by Uhlenbeck and Yau \cite{Ref:Uhl-Yau}.
	In this paper,
	we prove an analog of the Kobayashi-Hitchin correspondence on compact mini-holomorphic $3$-folds.
	
	We describe mini-holomorphic $3$-folds and mini-holomorphic vector bundles on them.
	In our context,
	they are the counterpart of complex manifolds and holomorphic bundles.
	Let $X$ be a compact oriented $3$-fold.
	Let $\partial_t$ be a nowhere vanishing vector field on $X$ and $\alpha$ a $\partial_t$-invariant $1$-form on $X$ with a condition $\ip<\partial_t,\alpha>=1$.
	Let $(\Sigma,g_{\Sigma})$ be an orbifold Riemann surface and $\pi:M\to\Sigma$ a $\partial_t$-invariant submersion.
	Set a metric $g=\alpha^2 + \pi^{\ast}g_{\Sigma}$.
	Assume that $\alpha\wedge\pi^{\ast}\mathrm{vol}_{\Sigma}$ is positive-oriented.
	We set $\Omega^{0,1}_X:= \underline{\mathbb{C}}\alpha\oplus\pi^{\ast}\Omega^{0,1}_{\Sigma}$ and $\Omega^{0,2}_X:=\bigwedge^2\Omega^{0,1}_X$,
	where $\underline{\mathbb{C}}\alpha$ is the subbundle spanned by $\alpha$.
	Then the tuple $(\partial_t,\alpha,\Sigma,\pi)$ is called a mini-holomorphic structure on $X$.
	We set a differential operator $\deebar_X(f):=\partial_t(f)\alpha + \tilde{\partial}_{\bar{z}}f\,\pi^{\ast}(d\bar{z})$,
	where $z$ is a holomorphic local chart of $\Sigma$ and $\tilde{\partial}_{\bar{z}}$ is the lift of $\partial_{\bar{z}}$ by the isomorphism $\mathrm{Ker}(\alpha)\simeq \pi^{\ast}T_{\mathbb{C}}\Sigma$.
	If a vector bundle $V$ on an open subset $U\subset X$ and a differential operator $\deebar_V:\Omega^{0,i}_X(V)\to\Omega^{0,i}_X(V)$ satisfies the Leibniz $\deebar_V(fs)=\deebar_X(f)\wedge s+f\deebar_V(s)\;(f\in C^{\infty}(U), s\in \Omega^{0,i}_X(V))$ and the integrability condition $\deebar_V\comp\deebar_V=0$,
	we call the tuple $(V,\deebar_V)$ a mini-holomorphic bundle on $U$.
	For a finite subset $Z\subset X$ and a mini-holomorphic bundle $(V,\deebar_V)$ on $X\setminus Z$,
	$(V,\deebar_V)$ is a Dirac-type singular mini-holomorphic bundle on $(X,Z)$ if it satisfies a certain condition (See Definition \ref{Def:alg. Dirac-type sing.}.).
	We describe HE-monopoles on mini-holomorphic $3$-folds.
	Let $(V,h,A)$ be a Hermitian vector bundle on an open subset $U\subset X$ with a connection.
	Let $\Phi$ be a skew-Hermitian endomorphism of $V$.
	The tuple $(V,h,A,\Phi)$ is HE-monopole of factor $c\in\mathbb{R}$ if it satisfies the Bogomolny-Hermite-Einstein equation $F(A) = \ast \nabla_A(\Phi) + \sqrt{-1}c\cdot(\pi^{\ast}\omega_{\Sigma})\mathrm{Id}_{V}$,
	where $\omega_{\Sigma}$ is the K\"{a}hler form of $\Sigma$.
	If the constant $c=0$,
	the Bogomolny-Hermite-Einstein equation agrees with the ordinary Bogomolny equation.
	A HE-monopole $(V,h,A,\Phi)$ on $X\setminus Z$ is a Dirac-type singular monopole if it satisfies a certain condition (See Definition \ref{Def:HE-monopole}.).
	A HE-monopole $(V,h,A,\Phi)$ has a natural mini-holomorphic structure $\deebar_V := \nabla_A^{0,1} - \sqrt{-1}\Phi\cdot\alpha$,
	and underlying mini-holomorphic bundles of Dirac-type singular HE-monopoles are Dirac-type singular.
	Conversely,
	for a mini-holomorphic bundle $(E,\deebar_E)$ and a Hermitian metric $h$ on $E$,
	there uniquely exist a connection $A_h$ and a skew-Hermitian endomorphism $\Phi_h$ on $E$ such that $\deebar_E = \nabla_{A_h}^{0,1} - \sqrt{-1}\Phi_h\cdot\alpha$.
	If the tuple $(E,h,A_h,\Phi_h)$ is a HE-monopole,
	we call $h$ a Bogomolny-Hermite-Einstein metric (or shortly BHE-metric).
	Moreover,
	if the tuple $(E,h,A_h,\Phi_h)$ can be lifted by the Hopf-fibration,
	then $h$ is called an admissible metric (See Definition \ref{Def:adm. met. on mini-hol.}.).
	A metric $h$ is an admissible BHE-metric if and only if the tuple $(E,h,A_h,\Phi_h)$ is a Dirac-type singular HE-monopole (See Proposition \ref{Prp:char. of adm.}.).
	We introduce the stability of Dirac-type singular mini-holomorphic bundles.
	We set the degree of $(E,\deebar_E)$ to be $\mathrm{deg}(E,\deebar_E):=\int_{X\setminus Z}\alpha\wedge c_1(A_{h_0})$,
	where $h_0$ is an admissible metric.
	By Proposition \ref{Prp:well-definedness of degree}, $\mathrm{deg}(E,\deebar_E)$ is independent of the choice of $h_0$, and hence it is well-defined.
	The Dirac-type singular mini-holomorphic bundle $(E,\deebar_E)$ is stable if the inequality 	$\mathrm{deg}(E,\deebar_E)/\mathrm{rank}(E) > \mathrm{deg}(F,\deebar_F)/\mathrm{rank}(F)$ holds for any proper mini-holomorphic subbundle $(F,\deebar_F)$ of $(E,\deebar_E)$.
	Our main result is the following:
	\begin{Thm}[Theorem \ref{Thm:ex. of BHE-met.}]
		If $(E,\deebar_E)$ is stable,
		then there exists an admissible BHE-metric $h$ on $E$.
	\end{Thm}

	In Section 2,
	we recall the notations necessary for Section 3.
	Moreover,
	in Proposition \ref{Prp:A criterion for Dirac-typeness} we give a slight generalization of Theorem 4.5 in \cite{Ref:Moc-Yos}.
	In Section 3,
	we prove our main result.
	
	\subsection*{Comparison with previous studies}
		In \cite{Ref:Cha-Hur},
		Charbonneau and Hurtubise introduced the notion of HE-monopoles and mini-holomorphic bundles on a product of $S^1$ and a Riemann surface $\Sigma$.
		They also proved the Kobayashi-Hitchin correspondence on $S^1\times \Sigma$.
		
		In \cite{Ref:Bis-Hur},
		Biswas and Hurtubise considered the Kobayashi-Hitchin correspondence on compact Sasakian $3$-folds,
		and proved that two Dirac-type singular monopole on a compact Sasakian $3$-folds are isomorphic as monopoles if their underlying mini-holomorphic structures are isomorphic.
		
		In \cite{Ref:Bar-Hek},
		Baraglia and Hekmati  constructed the Kobayashi-Hitchin correspondence for compact oriented taut Riemannian foliated manifolds with transverse Hermitian structure.
		This result seems to be considered as a higher-dimensional generalization of our result under the non-singular condition.
	\subsection*{Acknowledgments}
		I deeply thank my supervisor Takuro Mochizuki for his kind advices and discussions.
		I also thank Yoshinori Hashimoto for teaching me the result of \cite{Ref:Bar-Hek}.

%% file: Sec2.tex
\section{Preliminaries}
	\subsection{K\"{a}hler orbifolds}
		We recall the notion of orbifold by following \cite{Ref:Moe}.
		For a Lie groupoid $\mathcal{C}$,
		we denote by $\mathcal{C}_0$ and $\mathcal{C}_1$ the object space and the morphism space of $\mathcal{C}$.
		Let $s,t,m$ and $u$ be the source map, the target map, the composition map and the unit map.
		We set $\mathcal{C}(x,y):=(s,t)^{-1}(x,y)$.
		We also set $\mathcal{C}_2 := \mathcal{C}_{1}\fibp{s}{t} \mathcal{C}_{1}$, and denote by $p_i:\mathcal{C}_2 \to \mathcal{C}_1$ theprojection to the $i$-th component.
		
		\begin{Def}
			Let $\mathcal{C}$ be a Lie groupoid unless otherwise denoted.
			\begin{enumerate}[label=(\roman*)]
				\item 
					The groupoid $\mathcal{C}$ is called an orbifold if the following conditions are satisfied.
					\begin{itemize}
						\item
							The maps $s$ and $t$ are local diffeomorphisms.
						\item 
							The map $(s,t):\mathcal{C}_1\to \mathcal{C}_0\times \mathcal{C}_0$ is proper.
					\end{itemize}
					For an orbifold $\mathcal{C}$, we define the dimension of $\mathcal{C}$ as the one of $\mathcal{C}_0$.
				\item
					We denote by $|\mathcal{C}|$ the underlying topological space $\mathcal{C}_0/\sim$,
					where $x\sim y$ holds if $\mathcal{C}(x,y)\neq\emptyset$.
				\item 
					A vector bundle on $\mathcal{C}$ is a vector bundle $V$ on $\mathcal{C}_0$ equipped with an isomorphism $\phi:t^{\ast}V\to s^{\ast}V$ that satisfies the following commutative diagram:
					\[
					\begin{CD}
						p_2^{\ast}t^{\ast}V @>p_2^{\ast}\phi>> p_2^{\ast}s^{\ast}V @= p_1^{\ast}t^{\ast}V \\
						@| @. @VVp_1^{\ast}\phi V \\
						m^{\ast}t^{\ast}V @>m^{\ast}\phi>>m^{\ast}s^{\ast}V @= p_1^{\ast}s^{\ast}V 
					\end{CD}
					\] 
					The tangent and cotangent bundles on $\mathcal{C}_0$ naturally satisfy the above condition.
					In particular,
					a Riemannian metrics on $\mathcal{C}$ is a $\mathcal{C}_1$-invariant Riemannian metric on $\mathcal{C}_0$.
				\item
					An orbifold $\mathcal{C}$ is a complex orbifold if both $\mathcal{C}_0$ and $\mathcal{C}_1$ have complex structures such that $s$, $t$, $m$ and $u$ are holomorphic.
					Moreover,
					a complex orbifold $\mathcal{C}$ is called a K\"{a}hler orbifold if $\mathcal{C}$ equips a Riemannian metric $g$ such that $(\mathcal{C}_0,g)$ is K\"{a}hler.
				\item
					Let $M$ be a manifold and $\mathcal{C}$ an orbifold.
					A smooth map $\varphi$ from $M$ to $\mathcal{C}$ is a collection of an open covering $\{U_i\}_{i\in I}$ of $M$ and smooth maps $\varphi_i:U_i\to \mathcal{C}_0$ and $\varphi_{ij}:U_i\cap U_j \to \mathcal{C}_1$ that satisfies $(s,t)\comp \varphi_{ij} = (\varphi_j,\varphi_i)$ and $\varphi_{ij}\comp \varphi_{jk} = \varphi_{ik}$ for any $i,j,k\in I$.
					Moreover,
					a smooth map $\varphi:M\to\mathcal{C}$ is said to be a submersion if $\varphi_i$ is a submersion for any $i\in I$.
			\end{enumerate}
		\end{Def}
		For a preparation of following subsections,
		we show the following lemma.
		\begin{Lem}
			Let $f:M\to\mathcal{C}$ be a smooth map from a manifold $M$ to an orbifold $\mathcal{C}$, and $\{U_i\}_{i\in I}$ the associated open covering.
			Let $V$ be a vector bundle on $\mathcal{C}$.
			Then we have the isomorphisms $f_j^{\ast}V|_{U_i\cap U_j} \simeq f_i^{\ast}V|_{U_i\cap U_j}$ that satisfies the cocycle condition. 
		\end{Lem}
		\begin{proof}
			Obvious from the definition of vector bundles on a Lie groupoid.
		\end{proof}
		\begin{Def}
			Let $f:M\to\mathcal{C}$ and $V$ be as in the above lemma.
			We define the pullback $f^{\ast}V$ as the gluing of each pullback $f_i^{\ast}V$.
		\end{Def}
		\begin{Remark}
			The pullback of a Riemannian metric is defined in a similar way. 
		\end{Remark}
	\subsection{The boundary value problem of HE-metrics on trivial holomorphic bundles}
		Let $(M,g)$ be a compact connected Hermitian manifold of dimension $n$ with boundary $\partial M$ and $\omega$ the fundamental form.
		Let $\Lambda_\omega:\Omega^{i}_M \to \Omega^{i-2}_M$ be the contraction by $\omega$.
		We assume that $g$ is of $L^p_2$-class for some $p\gg 2n$.
		We denote by $\partial$ and $\deebar$ the $(1,0)$ and $(0,1)$-part of the exterior derivative.
		\begin{Def}
			Let $(V,\deebar_V)$ be a holomorphic bundle on $M$.
			A Hermitian metric $h$ on $(V,\deebar_V)$ is Hermite-Einstein of factor $f\in L^1(M)$ if the Chern connection $A$ of $(V,\deebar_{V},h)$ satisfies the Hermite-Einstein equation $\Lambda_{\omega}F(A) = \sqrt{-1}f\,\mathrm{Id}_V$.
		\end{Def}
		For a holomorphic line bundle $(L,\deebar_L)$ on $M$,
		a Hermitian metric $h$ on $L$ is Hermite-Einstein of factor $f\in L^1(M)$ if and only if we have $\tilde{\Delta}(\log(h))=-f$,
		where $\tilde{\Delta}:=\sqrt{-1}\Lambda_{\omega}\deebar\partial$ is the complex Laplacian. 
		\begin{Lem}\label{Lem:CpxLap}
			For any $s\in L^p(M)$ there uniquely exists $f\in L^p_2(M)$ such that $\tilde{\Delta}(f)=s$ and  $f|_{\partial \Omega}=0$.
		\end{Lem}
		\begin{proof}
			Set the operator $\hat{\Delta} := \partial^{\ast_g}\partial$,
			where $\partial^{\ast_g}$ is the adjoint of $\partial$ with respect to $g$.
			Then the operator $\hat{S}:L^p_2(M) \ni f \mapsto (\hat{\Delta}(f),f|_{\partial M})\in L^p(M)\oplus L^{p}_{2-(1/p)}(\partial \Omega)$ is an isomorphism by the Lax-Milgram argument.
			We set another operator $S:L^p_2(M) \ni f \mapsto (\tilde{\Delta}(f),f|_{\partial M})\in L^p(M)\oplus L^p_{2-(1/p)}(\partial \Omega)$.
			Then the difference $S-\hat{S}$ is a first-order differential operator with $C^0$-coefficient,
			and hence $S$ is a Fredholm operator of index $0$.
			Here we have $\mathrm{Ker}(S)=\{0\}$ by the maximum principle.
			Therefore $S$ is an isomorphism,
			which completes the proof.
		\end{proof}
		\begin{Remark}
			The non-integer order Sobolev space $L^p_k$ means the Sobolev-Slobodeckji space $W^{k,p}$, or equivalently the Besov space $B^k_{p,p}$.
		\end{Remark}
		Let $V\simeq \mathbb{C}^r \times M$ be a trivial holomorphic bundle on $M$.
		By following \cite{Ref:Don},
		we solve the Dirichlet problem of HE-metrics on $V$.
		\begin{Prp}\label{Prp:Bdd prblm of HE-metric}
			For any smooth Hermitian metric $\tilde{h}$ on $V|_{\partial M}$ and a real-valued function $f\in L^p(M)$,
			there exists a unique Hermitian metric $h$ on $V$ such that $h$ is a Hermite-Einstein metric of factor $f\in L^p(M)$ and satisfies $h|_{\partial M}=\tilde{h}$.  
		\end{Prp}
		\begin{proof}
			By Lemma \ref{Lem:CpxLap},
			we may assume $f=0$. 
			We first show the uniqueness.
			Let $h_1,h_2$ be Hermite-Einstein metrics on $V$ of factor $0$ such that $h_1|_{\partial M}=h_2|_{\partial M}=\tilde{h}$.
			We denote by $A_i$ the Chern connections of $(V,h_i)$ respectively.
			Set an endomorphism $\eta:=(h_1)^{-1}h_2$ on $V$.
			Then Hermite-Einstein condition induces $\deebar(\eta)\wedge\eta^{-1}\partial_{A_1}(\eta)=\deebar\partial_{A_1}(\eta)$.
			Hence by taking the trace and contraction by $\omega$,
			we obtain $\tilde{\Delta}(\mathrm{Tr}(\eta)) = -|\deebar_{V}\eta\cdot\eta^{-1/2}|_{h_1}^2\leq 0$ because $\eta$ is a Hermitian endomorphism with respect to $h_1$.
			The same argument applies to $\eta^{-1}$ and we also obtain $\tilde{\Delta}(\mathrm{Tr}(\eta^{-1}))\leq 0$.
			Therefore $\tilde{\Delta}(\mathrm{Tr}(\eta) + \mathrm{Tr}(\eta^{-1}))\leq 0$, and then maximum principle shows $\mathrm{max}_{M}(\mathrm{Tr}(\eta) + \mathrm{Tr}(\eta^{-1})) \leq \mathrm{max}_{\partial M}(\mathrm{Tr}(\eta) + \mathrm{Tr}(\eta^{-1})) = 2r$.
			Since we have $(\mathrm{Tr}(\eta) + \mathrm{Tr}(\eta^{-1}))\geq 2r$ by some calculation,
			the equality $\mathrm{Tr}(\eta) + \mathrm{Tr}(\eta^{-1}) = 2r$ holds identically.
			Therefore we obtain $\eta=\mathrm{Id}_V$, and which proves the uniqueness.
			
			We prove the existence by the method of continuity.
			Since the space of smooth Hermitian metric on $V|_{\partial M}$ is contractible,
			there exists a smooth family $\{\tilde{h}_t\}_{t\in[0,1]}$ of smooth Hermitian metrics on $V|_{\partial M}$ such that $\tilde{h}_0$ is the trivial metric and $\tilde{h}_1 = \tilde{h}$.
			We set $I\subset[0,1]$ to be the set consisting of $t\in[0,1]$ such that the solution $h$ exists in the case $\hat{h}=\hat{h}_t$.
			Obviously $0\in I$ and particularly $I\neq\emptyset$.
			We  prove that $I$ is open.
			We fix an arbitrary $s\in I$.
			Let $X_i$ be the space of $L^2_i$-valued Hermitian endomorphism on $V$ with respect to $h_{s}$.
			Let $Y_i$ be the space of $L^2_{i}$-valued Hermitian endomorphism on $V|_{\partial M}$ with respect to $h_{s}|_{\partial M}$.
			Take a positive number $k > n$.
			Let $O\subset X_k$ be a neighborhood of $\mathrm{Id}_V$ such that $h_{s}e$ is a Hermitian metric for any $e\in O$.
			We set an operator $\Xi:O\to X_{k-2}\times Y_{k-1/2}$ to be $\Xi(e):= (\sqrt{-1}\Lambda_{\omega}F(A_{h_{s}e}),e|_{\partial M})$.
			Then we have $(d\,\Xi)|_{\mathrm{Id}_V}(v) = (\tilde{\Delta}_{h_{s}}(v),v|_{\partial M})$,
			where $\tilde{\Delta}_{h_{s}} := \sqrt{-1}\Lambda_{\omega}\deebar\partial_{A_{h_{s}}}$.
			By the same argument in Lemma \ref{Lem:CpxLap},
			$(d\,\Xi)|_{\mathrm{Id}_V}$ is an isomorphism.
			Therefore the implicit function theorem shows that $s$ is an interior point of $I$.
			We prove that $I$ is closed.
			Let $\{h_i\}_{i\in\mathbb{N}}$ be a sequence of Hermite-Einstein metrics on $V$ such that $\{h_i|_{\partial M}\}$ converges to a smooth Hermitian metric $\tilde{h}_{\infty}$ in the sense of $C^1$-norm.
			We introduce the Donaldson metric on Hermitian metric space.
			For Hermitian metrics $H_1,H_2$ on $\mathbb{C}^r$,
			we set $\sigma(H_1,H_2):= \mathrm{Tr}(H_1^{-1}H_2) + \mathrm{Tr}(H_2^{-1}H_1)- 2r$.
			Then $\sigma$ is a complete metric on the space of Hermitian metrics and its topology coincides with the induced one from the linear space of skew-Hermitian forms on $\mathbb{C}^r$.
			We consider functions $d_{ij}:=\sigma(h_i,h_j)$ on $M$.
			By the same argument in the uniqueness part,
			we have $\tilde{\Delta}(d_{ij}) \leq 0$,
			and hence $\mathrm{max}_{M}(d_{ij}) = \mathrm{max}_{\partial M}(d_{ij})$.
			Since $\{h_i|_{\partial M}\}$ converges to $\tilde{h}_{\infty}$ in $C^0$-sense,
			there exists a $C^0$-limit $h_{\infty}$ of the sequence $\{h_i\}$.
			By the elliptic regularity and the (vacuum) Hermite-Einstein equation $\Lambda_{\omega}\deebar(h^{-1}\partial h)=0$, it is sufficient to prove the regularity of $h_{\infty}$ that the $C^1$-norms of $h_i$ are bounded.
			We assume that the norms $||h_i||_{C^1}$ are unbounded.
			Without loss of generality,
			we may assume $||h_i||_{C^1} \to \infty$.
			For $i\in\mathbb{N}$,
			we take $p_i \in M$ to satisfy $m_i := |dh_i|(p_i) \to \infty$.
			First we consider the case $\delta =\underline{\lim}_i\;\mathrm{dist}(p_i,\partial M)\cdot m_i > 0$.
			For $i\in \mathbb{N}$,
			we set $K_i$ to be the rescaling of $h_i|_{B(p_i,\delta/m_i)}$ to $B(0,\delta)\subset \mathbb{C}^n$,
			where $B(p,\eps):=\{z\in\mathbb{C}^n \mid |z-p|\leq\eps\}$.
			Since $\{h_i\}$ converge to $h_{\infty}$ in $C^0$-sense,
			$\{K_i\}$ converges to a constant metric $K_{\infty}$ in $C^0$-sense.
			We have $||dK_i||_{C^0} = 1$ and each $K_i$ is a Hermite-Einstein metric on $B(0,\delta)$ with respect to each rescaled Hermitian metric.
			Hence $\{K_i\}$ converges to the constant metric $K_{\infty}$ in $C^1$-sense.
			However, this contradict to the assumption $|dK_i|(0)=|dh_i|(p)/m_i = 1$.
			Next we consider the case $\underline{\lim}_i\;\mathrm{dist}(p_i,\partial M)\cdot m_i = 0$.
			We use closed half balls instead of closed balls,
			and then a similar argument induces a contradiction.
			Therefore the $C^1$-norms of $h_i$ are bounded and this completes the proof.
		\end{proof}
	\subsection{$3$-dimensional Sasakian manifolds}
		By following \cite{Ref:Boy-Cha},
		we recall the notion of Sasakian manifolds.
		\begin{Def}
			Let $(X,g)$ be a $2n+1$-dimensional Riemannian manifold.
			If there exists a $\mathbb{R}_{+}$-invariant complex structure $J$ on $\mathbb{R}_{+}\times X$ such that $(\mathbb{R}_{+}\times X,J, dr^2 + r^2g)$ is K\"{a}hler,
			then the tuple $(X,g,J)$ is called a Sasakian manifold.
			For a Sasakian manifold  $(X,g,J)$,
			the Killing vector field $\xi:= -J(\partial_r)|_{\{1\}\times X}$ on $X$ is called \textbf{the Reeb vector field} on $(X,g,J)$.
		\end{Def}
		Let $(X,g,J)$ be a Sasakian manifold.
		If any orbits of the Reeb vector field of $(X,g,J)$ is compact,
		then $(X,g,J)$ is called a quasi-regular Sasakian manifold.
		For a quasi-regular Sasakian manifold $(X,g,J)$,
		the Reeb vector field determines an almost free $S^1$-action on $X$.
		In particular,
		$X/S^1(= (\mathbb{R}_{+}\times X) / \mathbb{C}^{\ast})$ is Morita-equivalent to a complex orbifold.
		For a compact $3$-dimensional Sasakian manifold,
		we have the following result in \cite{Ref:Boy-Cha}.
		\begin{Thm}
			Any compact $3$-dimensional Sasakian manifolds are quasi-regular. 
		\end{Thm}
	\subsection{The Mini-holomorphic Structure}
		Let $X$ be a $2n+1$ dimensional manifold.
		\begin{Def}
			Let $\partial_{t}$ be a nowhere vanishing vector field on $X$.
			Let $\alpha\in \Omega^{1}_X$ be a $\partial_{t}$-invariant $1$-form such that $\ip<\alpha,\partial_{t}>=1$.
			Let $\pi:X\to Y$ be a $\partial_{t}$-invariant submersion to an $n$-dimensional complex orbifold $Y$.
			For a local vector field $\nu$ on $Y_0$,
			we will denote by $\widetilde{\nu}$ the lift of $\nu$ by the isomorphism $\mathrm{Ker}(\alpha) \simeq \pi^{\ast}TY$.
			\begin{itemize}
				\item 
					The tuple $(\partial_{t},\alpha,Y,\pi)$ is called an almost mini-holomorphic structure on $X$.
				\item
					We define $\Omega^{0,1}_X := \pi^{\ast}\Omega^{0,1}_{Y}\oplus \underline{\mathbb{C}}\alpha$ and $\Omega^{0,i}_X := \bigwedge^{i}\Omega^{0,1}_X$,
					where $\underline{\mathbb{C}}\alpha$ means a subbundle of $\Omega^{1}_{X,\mathbb{C}}$ spanned by $\alpha$.
				\item
					An almost mini-holomorphic structure $(\partial_{t},\alpha,Y,\pi)$ is called a mini-holomorphic structure on $X$ if we have $(d\alpha)^{0,2}=0$,
					where $(d\alpha)^{0,2}$ is the $\Omega^{0,2}_X$-component of $d\alpha$ with respect to the decomposition $\Omega^{2}_{X,\mathbb{C}} = \Omega^{0,2}_X \oplus (\Omega^{0,1}_X\otimes\pi^{\ast}\Omega^{1,0}_Y) \oplus \pi^{\ast}\Omega^{2,0}_Y$.
				\item 
					We define a differential operator $\deebar_X:\Omega^{0,0}_X \to \Omega^{0,1}_X$ to be $\deebar_X(f) := \partial_{t}(f)\alpha + \sum_i\widetilde{\partial_{\bar{z}^i}}(f)\pi^{\ast}(d\bar{z}^i)$,
					where $(z^i)$ is a holomorphic local chart on $Y_0$.
					Moreover,
					we extend $\deebar_X$ to be an operator from $\Omega^{0,i}_X$ to $\Omega^{0,i+1}_X$ by a usual way.
				\item
					If a function $f$ locally defined on $X$ satisfies $\deebar_X(f)=0$,
					then $f$ is said to be a mini-holomorphic function.
			\end{itemize}
		\end{Def}
		\begin{Remark}\ \,
			\begin{itemize}
				\item
					If $Y$ is a complex manifold and $X=\mathbb{R}_t\times Y$ or $S_t^1\times Y$,
					then there exists the trivial mini-holomorphic structure $(\partial_t,dt,Y,p:X\to Y)$,
					where $p$ is the projection.
					On the trivial mini-holomorphic structure,
					the notion of mini-holomorphic bundle as above defined agree with the one in \cite{Ref:Moc-Yos}.
				\item 
					For a mini-holomorphic structure $(\partial_{t},\alpha,Y,\pi)$,
					we have $\deebar_X \comp \deebar_X=0$.
				\item
					For a nonempty open subset $U\subset X$,
					the tuple $(\partial_{t}|_U,\alpha|_U,\pi(U),\pi|_U)$ is a mini-holomorphic structure on $U$.
			\end{itemize}
		\end{Remark}
		For example,
		Any principal $S^1$-bundle on a Riemann surface with a connection has a unique mini-holomorphic structure.
		
		Let $(M,g_M)$ be a compact Sasakian $3$-fold and $\partial_{t}$ the Reeb vector field on $M$.
		Let $\Sigma$ be the orbifold Riemann surface obtained as the quotient of $M$ by the $S^1$-action induced by $\partial_{t}$, and $\pi:M\to\Sigma$ the quotient map.
		Let $\alpha\in \Gamma(M,(\pi^{\ast}\Omega^1_{\Sigma})^{\bot})$ be the unique section that satisfies $\ip<\alpha,\partial_{t}>=1$,
		where $(\pi^{\ast}\Omega^1_{\Sigma})^{\bot}$ is the orthogonal complement bundle of $\pi^{\ast}\Omega^1_{\Sigma}$ in $\Omega^1_M$.
		\begin{Prp}
			The tuple $(\partial_{t},\alpha,\Sigma,\pi)$ is a mini-holomorphic structure on $M$.
		\end{Prp}
		\begin{proof}
			Any conditions other than the $\partial_{t}$-invariance of $\alpha$ is trivial,
			and $\alpha$ is $\partial_{t}$-invariant because $\mathrm{rank}((\pi^{\ast}\Omega^1_{\Sigma})^{\bot})=1$.
		\end{proof}
		\begin{Def}
			Let $(\partial_{t},\alpha,Y,\pi)$ be a mini-holomorphic structure on $X$.
			\begin{enumerate}[label=(\roman*)]
				\item 
					Let $V$ be a complex vector bundle on $X$.
					A differential operator $\deebar_{V}: \Omega_X^{0,i}(V)\to \Omega_X^{0,i+1}(V)$ is said to be a mini-holomorphic structure on $V$ if the following conditions are satisfied:
					\begin{itemize}
						\item 
						The Leibniz rule $\deebar_{V}(fs) = \deebar_{X}(f)\wedge s + f\deebar_{V}(s)$ holds for any $f\in C^{\infty}(X)$ and $s\in\Omega_X^{0,i}(V)$.
						\item
						The integrability condition $\deebar_{V}\comp\deebar_{V}=0$ holds.
					\end{itemize}
					We call the pair $(V,\deebar_{V})$ a mini-holomorphic bundle on $X$.
				\item
					If a local section $v$ on $V$ satisfies $\deebar_V(v)=0$,
					then we call $v$ a mini-holomorphic section.
				\item
					Let $(V_i,\deebar_{V_i})$ be a mini-holomorphic bundle on $X$ for $i=1,2$.
					A homomorphism $\phi:V_1\to V_2$ is mini-holomorphic if we have $\deebar_{V_2}\comp\phi = \phi\comp\deebar_{V_1}$.
					Moreover, if $\phi$ is an injective homomorphism of vector bundles,
					then $V_1$ is called a mini-holomorphic subbundle of $V_2$.
			\end{enumerate}
		\end{Def}
		Let $(\partial_{t},\alpha,Y,\pi)$ be a mini-holomorphic structure on $X$ and $(V,\deebar_{V})$ a mini-holomorphic bundle on $X$.
		Let $U\subset X$ be a sufficiently small open subset such that $\pi$ can be lifted to $\pi_U:U\to Y_0$.
		Let $W\subset \pi_U(U)$ be an open subset. 
		A smooth map $s:W\to X$ is called a section on $W$ \textit{i.e.}, $s$ satisfies $\pi_U\comp s = \mathrm{Id}_{W}$.
		For a section $s:W\to X$,
		the pullback $s^{\ast}V$ has a natural holomorphic structure $\deebar_{s^{\ast}V}:\Omega^{0,0}_{W}(s^{\ast}V) \to \Omega^{0,1}_{W}(s^{\ast}V)$ defined as follows:
		\[
			\deebar_{s^{\ast}V}(v) := \deebar_{V}(\tilde{v})|_{s(W)}
			\ \ \ \ \ \ \ \ (v\in \Omega^{0,0}_{W}(s^{\ast}V)),
		\]
		where $\tilde{v}$ is the local section of $V$ on a neighborhood of $s(W)$ that is obtained by the parallel transport of $v$ with the differential equation $\ip<\partial_{t},\deebar_V(\tilde{v})>=0$ along each integral curve of $\partial_t$.
		
		We define the scattering map of a mini-holomorphic bundle in our context.
		Since the following arguments are local with respect to $X$ and $Y$,
		we assume that $Y$ is a domain of $\mathbb{C}^n$ and $X=(-\eps,\eps)_t \times Y$.
		Let $s_1,s_2$ be a section on $Y$.
		We set the scattering map $\Psi_{s_1,s_2}:(s_1)^{\ast}V \simeq (s_2)^{\ast}V$ to be $\Psi_{s_1,s_2}(v) := \tilde{v}|_{s_2(Y)}$,
		where $\tilde{v}$ is the parallel transport of $v$.
		The scattering map $\Psi_{s_1,s_2}$ is obviously an isomorphism of differentiable vector bundle.
		\begin{Prp}[\cite{Ref:Cha-Hur}]
			The scattering map $\Psi_{s_1,s_2}$ is a holomorphic isomorphism.
		\end{Prp}
		\begin{proof}
			Let $v$ be a local holomorphic section of $(s_1)^{\ast}V$.
			By the integrability condition $\deebar_V\comp \deebar_{V}=0$, the parallel transport $\tilde{v}$ satisfies $\deebar_V(\tilde{v})=0$.
			Therefore $\Psi_{s_1,s_2}(v) = (s_2)^{\ast}(\tilde{v})$ is a holomorphic section of $(s_2)^{\ast}V$.
		\end{proof}
		
		In the following part of this subsection,
		we assume $\mathrm{dim}(X)=3$.
		We define the notion of algebraic Dirac-type singularities of mini-holomorphic bundle on $X$.
		Similar in the above argument,
		we assume that $Y$ is a neighborhood of $p\in\mathbb{C}$ and $X=(-\eps,\eps)_t \times Y$.
		Let $(V,\deebar_{V})$ be a mini-holomorphic bundle on $X\setminus\{(0,p)\}$.
		Set sections $s_1,s_2:Y \to X$ as $s_i(z) := ((-1)^{i}(\eps/2),z)$.
		\begin{Def}\label{Def:alg. Dirac-type sing.}
			\;\,\ 
			\begin{itemize}
				\item 
					The point $(0,p)$ is an algebraic Dirac-type singularity of $(V,\deebar_V)$ if the scattering map $\Psi_{s_1,s_2}$ can be prolonged to the meromorphic isomorphism $(s_1)^{\ast}V(\ast \pi(p)) \simeq (s_2)^{\ast}V(\ast \pi(p))$.
					Moreover,
					The algebraic Dirac-type singularity $(0,p)$ is of weight $\vec{k}=(k_i)\in\mathbb{Z}^r$ if there exist holomorphic frames $\mb{e}_i=(e_{i,j})$ of $(s_i)^{\ast}V\;\;(i=1,2)$ such that we have $\Psi_{s_1,s_2}(e_{1,j})=z^{k_j}e_{2,j}$,
					where $z$ is a holomorphic local chart on $Y$ such that $z(p)=0$.
				\item
					Let $Z\subset X$ be a discrete subset.
					If each point $p\in Z$ is an algebraic Dirac-type singularity of a mini-holomorphic bundle $(E,\deebar_E)$ on $X\setminus Z$,
					then we call $(E,\deebar_E)$ a Dirac-type singular mini-holomorphic bundle on $(X,Z)$.
			\end{itemize}
		\end{Def}
		\begin{Remark}
			Since $\mathbb{C}\{z\}$ is a PID,
			the weight of an algebraic Dirac-type singularity is unique up to permutations. 
		\end{Remark}
	\subsection{The Hopf fibration and Dirac-type HE-monopoles}
		\subsubsection{The Hopf fibration}
			Let $U\subset \mathbb{R}^3$ be a neighborhood of $0\in\mathbb{R}^3$ and $g$ be a Riemannian metric on $U$.
			We assume that the canonical coordinate of $\mathbb{R}^3$ is the normal coordinate of $g$ at $0\in\mathbb{R}^3$.
			Set the Hopf fibration $p:\mathbb{R}^4=\mathbb{C}^2 \to\mathbb{R}^3=\mathbb{R}\times\mathbb{C}$ to be $p(z_1,z_2)=(|z_1|^2-|z_2|^2, 2z_1z_2)$, where we set $z_i=x_i+\sqrt{-1}y_i$.
			We also set the $S^1(= \mathbb{R}/2\pi\mathbb{Z})$-action on $\mathbb{C}^2$ to be $\theta\cdot(z_1,z_2) := (e^{\sqrt{-1}\theta}z_1,e^{-\sqrt{-1}\theta}z_2)$.
			Then the restriction $p:\mathbb{R}^4\setminus\{0\} \to \mathbb{R}^3\setminus\{0\}$ forms a principal $S^1$-bundle.
			\begin{Lem}\label{Lem:prepare of Pauly condition}
				There exist a harmonic function $f:U\setminus\{0\}\to\mathbb{R}$ with respect to the metric $g$ and a $1$-form $\xi$ on $p^{-1}(U)$ such that the following hold.
				\begin{itemize}
					\item
						The $1$-form $\omega:=p^{\ast}f\cdot\xi$ is a connection of $p:\mathbb{R}^4\setminus\{0\} \to \mathbb{R}^3\setminus\{0\}$,
						\textit{i.e.} $\omega$ is $S^1$-invariant, and we have $\omega(\partial_\theta)=1$.
						Here $\partial_\theta$ is the generating vector field of the $S^1$-action on $\mathbb{R}^4\setminus\{0\}$.
					\item
						We have $d\omega=-p^{\ast}(\ast df)$.
					\item 
						We have the following estimates:
						\begin{align*}
							\left\{
								\begin{array}{l}
									f = 1/2r_3 + o(1)\\
									\xi = 2(-y_1dx_1 + x_1dy_1 + y_2dx_2 - x_2dy_2 ) + O(r^2_4).\\
								\end{array}
							\right.
						\end{align*}
					\item
						The symmetric tensor $g_4=p^{\ast}f(p^{\ast}g + \xi^2)$ is a Riemannian metric of $L^2_{5,\mathrm{loc}}$-class on $p^{-1}(U)$,
						and we have an estimate $|g_4-2g_{4,\mathrm{Euc}}|_{g_{4,\mathrm{Euc}}}=O(r_4)$.
						Here a function on $p^{-1}(U)$ is of $L^2_{k,\mathrm{loc}}$-class if every derivative of $f$ up to order $k$ has a finite $L^2$-norm on any compact subset of $p^{-1}(U)$.
				\end{itemize}
			\end{Lem}
			\begin{Remark}\ \,
				\begin{itemize}
					\item 
						If $g=g_{\mathbb{R}^3}$,
						we can choose $f=1/2r_3$ and $\xi = 2(-y_1dx_1 + x_1dy_1 + y_2dx_2 - x_2dy_2 )$.
						Then we have $g_4 = 2g_{\mathbb{R}^4}$.
					\item 
						By the Sobolev embedding theorem,
						the connection matrix of $A_4$ is of $C^{3}$-class.
				\end{itemize}
			\end{Remark}
		\subsubsection{The Hopf fibration of mini-holomorphic manifolds}\label{subsubsec:Hopf-Fibration}
			Let $I=(-\eps,\eps)\subset \mathbb{R}_{t}$ and $W\subset \mathbb{C}_z$ be neighborhoods of origins of $\mathbb{R}$ and $\mathbb{C}$ respectively.
			Set $U:=I\times W$, and take the projection $\pi:U\to W$.
			Let $g_W = \lambda dzd\bar{z}$ be the K\"{a}hler metric on $W$ that satisfies $\lambda(0)=1$.
			Take an $\mathbb{R}_t$-invariant $1$-form $\alpha = dt + \alpha_x dx + \alpha_y dy\in\Omega^1_{U}$ satisfying $\alpha_x(0)=\alpha_y(0)=0$.
			Set $g_U := \alpha^2 + \pi^{\ast}g_W$.
			The tuple $(\partial_{t},\alpha,W,\pi)$ forms a mini-holomorphic structure on $U$.
			
			Take a normal coordinate $(x^i)$ at $(0,0)$ on $U$ satisfying $dt|_0 = dx^1|_0$ and $dz|_{0} = (dx^2 + \sqrt{-1}dx^3)|_0$.
			Let $U_4\subset \mathbb{R}^4$ be a sufficiently small neighborhood of $0\in \mathbb{R}^4$ and $p:U_4 \to U$ the Hopf fibration with respect to the coordinate $(x^i)$.
			We take a harmonic function $f$ on $U\setminus \{0\}$ and a $1$-form $\xi$ on $U_4$ as in Lemma \ref{Lem:prepare of Pauly condition}.
			Set an $L^2_{5,\mathrm{loc}}$-metric $g_4 := p^{\ast}(p^{\ast}g_U + \xi^2)$ on $U_4$.
			Set an almost complex structure $J\in \Gamma(U_4,\mathrm{End}(TU_4))\simeq \Gamma(U_4\setminus \{0\},\mathrm{End}(\Omega^1_{U_4}))$ to be 
			\[
				\left\{
					\begin{array}{ll}
						J(-\xi) &:= -p^{\ast}\alpha\\
						J(p^{\ast}dx) &:= -p^{\ast}dy.
					\end{array}
				\right.
			\]
			\begin{Lem}
				The almost complex structure $J$ is integrable and of $L^p_{2,\mathrm{loc}}$-class on $U_4$ for any $p\in[1,\infty)$.
			\end{Lem}
			\begin{proof}
				We can check vanishing of the Nijenhuis tensor of $J$ by an easy calculation,
				and hence $J$ is integrable.
				Set a $1$-form $\xi_0:=2(-y_1dx_1 + x_1dy_1 + y_2dx_2 - x_2dy_2 )$ and take an almost complex structure $J_0$ on $U_4$ to be 
				\[
					\left\{
						\begin{array}{ll}
						J_0(-\xi_0) &:= -p^{\ast}dx^1\\
						J_0(p^{\ast}dx^2) &:= -p^{\ast}dx^3.
					\end{array}
					\right.
				\]
				Then $J_0$ agrees with the canonical complex structure on $\mathbb{R}^4$.
				Hence $J_0$ is of $C^{\infty}$-class on $U_4$.
				By some calculation the difference $J_0 - J$ is of $L^p_{2,\mathrm{loc}}$-class for any $p\in[1,\infty)$, which proves the lemma.
			\end{proof}
			By the result in \cite{Ref:Hil-Tay},
			we obtain the following corollary.
			\begin{Cor}\label{Cor:hol. coord. on Hopf lift}
				There exists a holomorphic coordinate $(w_1,w_2)$ on $(U_4,J)$ such that satisfies the following:
				\begin{itemize}
					\item 
						The coordinates $w_1$ and $w_2$ are of $L^p_{3,\mathrm{loc}}$-class for any $p\in[1,\infty)$.
					\item 
						We have $\theta\cdot w_1 = e^{\sqrt{-1}\theta}w_1$ and $\theta\cdot w_2 = e^{-\sqrt{-1}\theta}w_2$ for any $\theta\in S^1$.
					\item
						The equality $p^{\ast}z=w_1w_2$ holds.
				\end{itemize}
			\end{Cor}
			\begin{proof}
				Since the $S^1$-weight of $\Omega^{1,0}_{\mathbb{C}^2}|_{0}$ is $(1,-1)$.
				Hence we can take a holomorphic local chart $w_1,\,w_2$ that $\theta\cdot w_1 = e^{\sqrt{-1}\theta}w_1$ and $\theta\cdot w_2 = e^{-\sqrt{-1}\theta}w_2$ for any $\theta\in S^1$.
				Since $p^{\ast}z$ is a holomorphic function and $S^1$-invariant and of order $2$ at origin,
				we may assume $p^{\ast}z=w_1\cdot w_2$.
			\end{proof}
			Set $U_{+} := U \setminus ((-\eps,0] \times \{0\})$ and $U_{-}:= U \setminus ([0,\eps)\times \{0\})$.
			For Proposition \ref{Prp:the lift of mini-hol. bundle},
			we prepare the following lemma.
			\begin{Lem}\label{Lem:pullback of zero locus}
				We have $\mathrm{Zero}(w_1) = p^{\ast}((-\eps,0]\times \{0\})$ and $\mathrm{Zero}(w_2) = p^{\ast}([0,\eps)\times \{0\})$.
			\end{Lem}
			\begin{proof}
				We have $p^{\ast}((-\eps,\eps)\times \{0\}) = \mathrm{Zero}(p^{\ast}z)=\mathrm{Zero}(w_1)\cup \mathrm{Zero}(w_2)$ because $p^{\ast}z = w_1w_2$.
				Hence $p^{\ast}((-\eps,0]\times \{0\} \sqcup [0,\eps)\times \{0\}) = (\mathrm{Zero}(w_1)\setminus\{(0,0)\}) \sqcup (\mathrm{Zero}(w_2)\setminus\{(0,0)\})$.
				Since the image of a connected space by a continuous map is connected,
				we obtain $\mathrm{Zero}(w_1) = p^{\ast}((-\eps,0]\times \{0\})$ and $\mathrm{Zero}(w_2) = p^{\ast}([0,\eps)\times \{0\})$, or $\mathrm{Zero}(w_1) = p^{\ast}([0,\eps)\times \{0\})$ and $\mathrm{Zero}(w_2) = p^{\ast}((-\eps,0]\times \{0\})$.
				We assume the latter one.
				Let $L$ be an $S^1$-equivariant line bundle on $U_4$ such that the weight of $L|_{0}$ is $1$.
				Let $b$ be a frame of $L$ such that $\theta\cdot b = e^{\sqrt{-1}\theta}b$.
				We denote by $\tilde{L}$ the quotient of $L|_{U_4 \setminus\{0\}}$.
				On one hand,
				for the canonical coordinate $z_1,z_2$ on $\mathbb{C}^2$,
				the descent of the sections $(z_1)^{-1}\cdot b$ on $U_4 \setminus \mathrm{Zero}(z_1)$ and $z_2\cdot b$ on $U_4 \setminus \mathrm{Zero}(z_2)$ give frames of $\tilde{L}$ on $U_{+}$ and $U_{-}$ respectively.
				By using this frame, the degree of $\tilde{L}$ is calculated as $1$.
				On the other hand,
				the descent of the frame $(w_1)^{-1}\cdot b$ on $U_4 \setminus \mathrm{Zero}(w_1)$ and $w_2\cdot b$ on $U_4 \setminus \mathrm{Zero}(w_2)$ also gives frames of $\tilde{L}$ on $U_{-}$ and $U_{+}$ respectively.
				Then the degree of $\tilde{L}$ is calculated as $-1$,
				which is a contradiction.
				Therefore we obtain $\mathrm{Zero}(w_1) = p^{\ast}((-\eps,0]\times \{0\})$,  $\mathrm{Zero}(w_2) = p^{\ast}([0,\eps)\times \{0\})$.
			\end{proof}
		
			\begin{Remark}\ \,
				\begin{itemize}
					\item 
						If $p>4$,
						we have $L^p_i \subset C^{\,i-1,(p-4)/p}$.
					\item
						The tuple $(U_4,J,g_4)$ is not a K\"{a}hler manifold in general.
				\end{itemize}
			\end{Remark}
			
			We mention the lift of Dirac-type singular mini-holomorphic bundles by the Hopf fibration.
			Let $(E,\deebar_E)$ be a Dirac-type singular mini-holomorphic bundle on $(U,\{(0,0)\})$ of weight $\vec{k}=(k_i)\in\mathbb{Z}^r$.
			We set a holomorphic structure $\deebar_{E_4}:\Omega^{0,0}(E_4)\to \Omega^{0,1}(E_4)$ on $E_4 := p^{\ast}E$ as follows: \[\deebar_{E_4}(p^{\ast}v):= (p^{\ast}\deebar_{E_4}(v))^{0,1},\]
			where $v$ is a local section of $E$.
			Take mini-holomorphic frames $\mb{e}_{\pm}=(e_{\pm,i})$ of $E$ on $U_{\pm}$ such that we have $e_{+,i}=z^{k_i}\cdot e_{-,i}$.
			
			By Lemma \ref{Lem:pullback of zero locus},
			we get a frame $\mb{e}_4 = (e_{4,i})$ of $E_4$ on $U_4 \setminus\{(0,0)\}$ given as follows:
			\[
				e_{4,i}=
				\left\{
				\begin{array}{ll}
					w_1^{k_i}\cdot p^{\ast}(e_{+,i})  &\;(w_1 \neq 0)\\
					w_2^{-k_i}\cdot p^{\ast}(e_{-,i}) &\;(w_2 \neq 0).
				\end{array}
				\right.
			\]
			We extend $E_4$ over $U_4$ by this frame.
			Then the weight of $E_4|_{0}$ is $\vec{k}=(k_i)\in\mathbb{Z}^r$ because $\theta \cdot e_{4,i} = e^{\sqrt{-1}k_i}\cdot e_{4,i}$.
			Summarizing the above argument,
			we obtain the following proposition.
			\begin{Prp}\label{Prp:the lift of mini-hol. bundle}
				For the Dirac-type singular mini-holomorphic bundle $(E,\deebar_E)$ of weight $\vec{k}\in\mathbb{Z}^r$ and the mini-holomorphic frames $\mb{e}_{\pm}$ on $U_{\pm}$,
				the lift $E_4:=p^{\ast}E$ has the natural $S^1$-equivariant prolongation over $U_4$ and the $S^1$-weight of $E_4|_{0}$ is $\vec{k}$.
			\end{Prp}
			For the proof of Theorem \ref{Thm:ex. of BHE-met.},
			we prove the following lemma.
			\begin{Lem}\label{Lem:rest. of mini-hol. is regular}
				Let $(E,\deebar_E)$ be a Dirac-type singular mini-holomorphic bundle on $(U,\{(0,0)\})$ of rank $r$ and $(F,\deebar_F)$ be a mini-holomorphic subbundle of $(E,\deebar_E)$ of rank $r'$.
				Then the lift $F_4$ of $F$ in Proposition \ref{Prp:the lift of mini-hol. bundle} is a holomorphic subbundle of the lift $E_4$ of $E$.
			\end{Lem}
			\begin{proof}
				Let $\mb{e}=(e_{\pm,i})$ and $\mb{f}_{\pm}=(f_{\pm,j})$ be the mini-holomorphic frames of $E$ and $F$ on $U_{\pm}$ respectively such that there exist $\vec{k}=(k_i)\in\mathbb{Z}^r$ and $\vec{k'}=(k'_{i})\in\mathbb{Z}^{r'}$ such that $e_{+,i}=\sum_iz^{k_i}\cdot e_{-,i}$ and $f_{+,j}=z^{k'_j}\cdot f_{-,j}$ for any $i,j$.
				By shrinking $U$ if necessary,
				we may assume that for any $j\in\{1,\dots,r'\}$ there exist mini-holomorphic functions $(a^i_j)_{i=1}^r$ on $U_+$ such that $f_{j,+}=a^i_j\cdot e_{i.+}$.
				By the definition of mini-holomorphic functions,
				$a^i_j$ can be prolonged to whole $U$ uniquely.
				Since we have $f_{j,-}=\sum_i a^i_j\,z^{k_i-k'_j}\,e_{-,i}$,
				we obtain $a^i_j=0$ unless $k_i=k'_j$.
				Therefore we obtain $f_{j,\pm}=\sum_i a^i_je_{\pm,i}$.
				Let $\mb{e}_4=(e_{4,i})$ and $\mb{f}_4=(f_{4,j})$ be the holomorphic flames of $E_4$ and $F_4$ that is used in construction of $E_4$ and $F_4$ respectively.
				Then we have $f_{4,j}=\sum_i p^{\ast}(a^i_j)\cdot e_{4,i}$.
				Since $F_4|_{U_4\setminus\{(0,0)\}}$ is a subbundle of $E_4|_{U_4\setminus\{(0,0)\}}$,
				it suffices to show $\mathrm{rank}(a^i_j(0,0))=r'$.
				If we have $\mathrm{rank}(a^i_j(0,0))<r'$,
				then we have $\mathrm{rank}(a^i_j(-\eps,0))=\mathrm{rank}(a^i_j(0,0))<r'$ by the definition of mini-holomorphic functions.
				However,
				it contradicts to the assumption that $\mb{f}_{-}$ is a frame of $F|_{U_-}$.
			\end{proof}
		\subsubsection{HE-monopoles and underlying mini-holomorphic structures}
			Let $X$ be an oriented connected $3$-fold.
			Let $(\partial_{t},\alpha,\Sigma,\pi)$ be a mini-holomorphic structure on $X$.
			Let $g_{\Sigma}$ be a K\"{a}hler metric on $\Sigma$ and $\omega_{\Sigma}$ the K\"{a}hler form of $(\Sigma,g_{\Sigma})$.
			Set $g_X := \alpha^2 + \pi^{\ast}g_{\Sigma}$.
			We assume that $\alpha\wedge \pi^{\ast}\omega_{\Sigma}$ is positive-oriented.
			\begin{Def}\label{Def:HE-monopole}\ \,
				\begin{enumerate}[label=(\roman*)]
					\item
					Let $(V, h)$ be a Hermitian vector bundle with a unitary connection $A$ on $X$.
					Let $\Phi$ be a skew-Hermitian section of $\mathrm{End}(V)$.
					The tuple $(V, h, A, \Phi)$ is said to be a HE-monopole of degree $c\in\mathbb{R}$ on $X$ if it satisfies the Hermite-Einstein-Bogomolny equation $F(A) = \ast \nabla_A(\Phi) + \sqrt{-1}c\cdot(\pi^{\ast}\omega_{\Sigma})\mathrm{Id}_{V}$.
					\item\label{Enum_2: Def:def of monopoles}
					Let $Z \subset X$ be a discrete subset.
					Let $(V, h, A, \Phi)$ be a HE-monopole of rank $r \in \mathbb{N}$ on $X \setminus Z$.
					A point $p\in Z$ is called a Dirac-type singularity of the monopole $(V, h, A, \Phi)$ with weight $\vec{k}_p=(k_{p,i}) \in \mathbb{Z}^r$
					if the following holds.
					\begin{itemize}
						\item
						There exists a small neighborhood $B$ of $p$ such that
						$(V,h)|_{B\setminus\{p\}}$ is decomposed into a sum of Hermitian line bundles $\bigoplus_{i=1}^{r} F_{p,i}$
						with $\mathrm{deg}(F_{p,i})=\int_{\partial B} c_1(F_{p,i}) = k_{p,i}$.
						\item
						In the above decomposition, we have the following estimates,
						\begin{align*}
						\left\{
						\begin{array}{l}
						\Phi  = \displaystyle\frac{\sqrt{-1}}{2R_p}\sum_{i=1}^{r} k_{p,i}\cdot Id_{F_{p,i}} + O(1)\\
						\nabla_A(R_p\Phi) = O(1),
						\end{array}
						\right.
						\end{align*}
						where $R_p$ is the distance from $p$.
					\end{itemize}
					For a HE-monopole $(V,h,A,\Phi)$ on $X\setminus Z$,
					if each point $p\in Z$ is a Dirac-type singularity,
					then we call $(V,h,A,\Phi)$ a Dirac-type singular monopole on $(X,Z)$.
				\end{enumerate}
			\end{Def}
			In \cite{Ref:Pau},
			Pauly proved a characterization of Dirac-type singular monopoles using the Hopf fibration, and it remains valid for HE-monopoles.
			\begin{Thm}\label{Thm:Pauly's criterion}
				Let $U \subset \mathbb{R}_t\times\mathbb{C}_z$ be a neighborhood of $(0,0)$ and $(\partial_t,\alpha,W,\pi:U\to W)$ a mini-holomorphic structure on $U$.
				Let $(V,h,A,\Phi)$ be a HE-monopole on $U\setminus\{(0,0)\}$ of degree $c\in\mathbb{R}$.
				\begin{itemize}
					\item 
						The tuple $(V_4,h_4,A_4):=(p^{\ast}V,p^{\ast}h,p^{\ast}A-\xi\otimes p^{\ast}\Phi)$ is a Hermitian holomorphic bundle that satisfies the Hermite-Einstein condition of factor $c/p^{\ast}f$.
					\item
						The point $(0,0)$ is a Dirac-type singularity of the HE-monopole $(V,h,A,\Phi)$
						if and only if the tuple $(V_4,h_4,A_4)$ can be prolonged as $S^1$-invariant Hermitian holomorphic bundle over $U_4$.
						Moreover,
						the weight of $(V,h,A,\Phi)$ at $(0,0)$ agrees with the $S^1$-weight of $V_4|_0$.
				\end{itemize}
			\end{Thm}
			\begin{Remark}
				We have $(p^{\ast}f)^{-1}\in L^p_{1,\mathrm{loc}}(U_4)$ for any $p\in[1,\infty)$.
			\end{Remark}
			
			\begin{Prp}
				A HE-monopole $(V,h,A,\Phi)$ on $X$ has a natural mini-holomorphic structure $\deebar_V(v) := \nabla^{0,1}_A(v) -\sqrt{-1}\Phi(v)\alpha$.
			\end{Prp}
			\begin{proof}
				By a direct calculation.
			\end{proof}
			Let $(E,\deebar_E)$ be a mini-holomorphic bundle on $X$,
			and $h$ a Hermitian metric on $E$.
			As an analogue of the Chern connection,
			there uniquely exist a connection $A_h$ and a skew-Hermitian endomorphism $\Phi_h$ on $E$ such that $\deebar_V(v) := \nabla^{0,1}_{A_h}(v) -\sqrt{-1}\Phi_h(v)\alpha$.
			We call $A_h$ and $\Phi_h$ the Charbonneau-Hurtubise (or shortly CH) connection and endomorphism.
			If the tuple $(V,h,A_h\Phi_h)$ is a monopole on $X$,
			we call $h$ a Bogomolny-Hermite-Einstein (or shortly BHE) metric on $(V,\deebar_V)$.
			
			By following \cite{Ref:Cha-Hur},
			we mention the relation between Dirac-type singular monopole and mini-holomorphic bundles.
			We assume that $Y$ is a neighborhood of $0\in\mathbb{C}$ and $X=[-\eps,\eps]_t\times Y$.
			Let $(V,h,A,\Phi)$ be a HE-monopole on $X\setminus\{(0,0)\}$ and $(V,\deebar_V)$ the underlying mini-holomorphic bundle.
			We denote by $\vec{k}\in\mathbb{Z}^r$ the weight of $(V,h,A,\Phi)$ at $(0,0)$.
			Take sections $s_1,s_2$ on $Y$ to be $s_i(z) := ((-1)^{i}\eps,z)$.
			Let $\Psi_{s_1,s_2}:(s_1)^{\ast}V|_{Y\setminus\{0\}}\to (s_2)^{\ast}V|_{Y\setminus\{0\}}$ be the scattering map.
			\begin{Prp}
				The scattering map $\Psi_{s_1,s_2}$ induces a meromorphic isomorphism $(s_1)^{\ast}V(\ast 0)\simeq (s_2)^{\ast}V(\ast 0)$.
				In particular,
				$(0,0)$ is an algebraic Dirac-type singularity of $(V,\deebar_V)$.
				Moreover,
				the weight of the algebraic Dirac-type singularity of $(V,\deebar_V)$ at $(0,0)$ is $\vec{k}$.
			\end{Prp}
			\begin{proof}
				Let $(V_2,\deebar_2)$ be a Dirac-type singular mini-holomorphic bundle on $(U,\{(0,0)\})$ such that the weight at $(0,0)$ is $\vec{k}$.
				Take a small neighborhood $U_4\subset\mathbb{C}^2$ of $(0,0)\in\mathbb{C}^2$.
				Let $V_{4,1}$ be the holomorphic bundle on $U_4$ obtained by applying Theorem \ref{Thm:Pauly's criterion} to $(V,h,A,\Phi)$.
				Let $V_{4,2}$ be the holomorphic bundle on $U_4$ obtained by applying Proposition \ref{Prp:the lift of mini-hol. bundle} to $(V_2,\deebar_{V_2})$.
				Since the $S^1$-weights of $V_{4,1}$ and $V_{4,2}$ agree with each other,
				there exists an $S^1$-equivariant holomorphic isomorphism $K_4:V_{4,1}\to V_{4,2}$.
				Then the descent $K:V\to V_2$ is a mini-holomorphic isomorphism.
				Therefore the weights of algebraic Dirac-type singularity of $(V,\deebar_V)$ and $(V_2,\deebar_2)$ agrees with each other,
				which is the assertion of the Proposition.
			\end{proof}
	\subsection{A generalization of the Characterization of Dirac-type singularities in \cite{Ref:Moc-Yos}}
		Let $U=I\times W \subset \mathbb{R}_t\times \mathbb{C}_z,\;\alpha\in\Omega^1(U),\;g_{\Sigma},\;g_U$ and $J$ be as in subsubsection \ref{subsubsec:Hopf-Fibration}.
		We denote by $r_i:\mathbb{R}^i\to \mathbb{R}$ the distance function from the origin. 
		Let $(V,h,A,\Phi)$ be a HE-monopole on $U\setminus\{0\}$ of rank $r>0$ of factor $c\in\mathbb{R}$.
		The following proposition is a slight generalization of Theorem 4.5 in \cite{Ref:Moc-Yos}.
		\begin{Prp}\label{Prp:A criterion for Dirac-typeness}
			If the estimate $|\Phi|=O(r_3^{-1})$ is satisfied,
			then $(V,h,A,\Phi)$ is a Dirac-type monopole on $(U,\{(0,0)\})$.
		\end{Prp}
		\begin{proof}
			Take the Hopf-fibration $p:U_4 \to U$, $f:U\setminus\{0\}\to \mathbb{R}$, $\xi\in \Omega^1_{U_4}$ and the holomorphic coordinate $w_1,w_2$ on $(U_4,J)$ as in \ref{subsubsec:Hopf-Fibration}.
			We set $(V_4,h_4,A_4):= (p^{\ast}V,p^{\ast}h,p^{\ast}A + \xi\otimes p^{\ast}\Phi)$.
			Set $U_{+} := U \setminus ((-\eps,0]\times \{0\})$ and $U_{-} := U \setminus ([0,\eps)\times \{0\})$.
			Let $\deebar_V$ be the mini-holomorphic structure of $(V,h,A,\Phi)$.
			By the assumption,
			there exist $\vec{k}=(k_i)$ and mini-holomorphic frames $\mb{e}_{\pm}$ of $(V,\deebar_V)$ on $U_{\pm}$ such that we have $e_{+,i}= z^{k_i}e_{-,i}$ and the estimate $|e_{\pm,i}|_{h} = O(r_3^{-N})$ around the origin for some $N>1$. 
			We take the frame $\mb{e}_{4} =(e_{4,i})$ of $V_4 := p^{\ast}V$ on $U_4 \setminus \{0\}$ and prolong $V_4$ over $U_4$ as in Proposition \ref{Prp:the lift of mini-hol. bundle}.
			Then we obtain the estimate $|e_{4,i}|_{h_4} = O(r_4^{-N'})$ for some $N'>0$.
			we prepare the following lemma.
			\begin{Lem}\label{Lem:Removable singularity of HE-metric}
				Let $D\subset \mathbb{C}^2$ be a relatively compact neighborhood of $0\in\mathbb{C}^2$ with a smooth boundary $\partial D$.
				Let $(E,h_0)$ be a Hermitian holomorphic vector bundle on $\bar{D}=D\cup\partial D$ of rank $r$.
				Let $h$ be a HE-metric of factor $f\in L^p(\bar{U})$ on $E|_{\bar{D}\setminus\{0\}}$ for $p>8$.
				If the estimate $|e|_h = O(r_4^{-N})\cdot|e|_{h_0}$ holds for some positive number $N>0$ and for any local smooth section $e$ of $E$,
				then $h$ can be prolonged as a HE-metric of $E$ over whole $D$.
			\end{Lem}
			By the above Lemma \ref{Lem:Removable singularity of HE-metric},
			$h_4$ is at least of $C^1$-class.
			By Theorem \ref{Thm:Pauly's criterion} (ii),
			$(V,h,A,\Phi)$ is a Dirac-type singular monopole on $(U,\{0\})$.
		\end{proof}
		\begin{proof}[{\upshape \textbf{(The proof of Lemma \ref{Lem:Removable singularity of HE-metric})}}]
			We may assume $f=0$.
			Since the statement is local,
			we also may assume that $E$ is a trivial bundle.
			Moreover,
			we may assume that $h_0$ is a HE-metric and $h_0|_{\partial D}=h|_{\partial D}$ by Proposition \ref{Prp:Bdd prblm of HE-metric}.
			We set the endomorphism $k:= h_0^{-1}h$.
			Then $\log(\mathrm{Tr}(k))$ satisfies $\log(\mathrm{Tr}(k)+ \mathrm{Tr}(k^{-1}))|_{\partial\,U'} = \log(2r)$ and $\tilde{\Delta}(\log\mathrm{Tr}(k)+\log\mathrm{Tr}(k^{-1})) \leq 0$ on $D\setminus\{0\}$ by \cite[Lemma 3.1]{Ref:Sim}.
			We have an estimate $|\log\mathrm{Tr}(k)|,|\log\mathrm{Tr}(k^{-1})| = O(\log(r_4))$,
			and hence we obtain $\tilde{\Delta}(\log\mathrm{Tr}(k)+\log\mathrm{Tr}(k^{-1})) \leq 0$ on $D$ as a distribution.
			Therefore $\mathrm{Tr}(k)+\mathrm{Tr}(k^{-1})=2r$ by the maximum principle.
			Thus we obtain $k=\mathrm{Id}_E$, and particularly $h$ is smooth over $\bar{D}$.
		\end{proof}
    

%% file: Sec3.tex
\section{The K-H correspondence of Dirac-type singular mini-hol. bundles on compact mini-hol. $3$-folds.}
		\subsection{The flat lift of mini-hol. $3$-folds}
			Let $X$ be a $3$-fold with a mini-holomorphic structure $(\partial_t,\alpha,\Sigma,\pi)$.
			Let $g_{\Sigma}$ be a K\"{a}hler metric on $\Sigma$ and set $g_{X}:= \alpha^2 + \pi^{\ast}g_{\Sigma}$.
			We set $M := S^1_{\theta}\times X$ and $g_M := d\theta^2 + p^{\ast}g_X$,
			where $p:M\to X$ is the projection.
			Let $J$ be an almost holomorphic structure on $M$ such that
			\[
				\left\{
				\begin{array}{l}
					J(d\theta) = -\alpha\\
					J(\pi^{\ast}a) = \pi^{\ast}(J_{\Sigma}(a))\;\;\;(a\in \Omega^1_{\Sigma}),
				\end{array}
				\right.
			\]
			where $J_{\Sigma}$ is the complex structure on $\Sigma$.
			\begin{Prp}\label{Prp:Gauchonness of prod.}
				The almost complex structure $J$ is integrable and the tuple $(M,J,g_M)$ is a Gauduchon manifold.
			\end{Prp}
			\begin{proof}
				The integrability is trivial from an easy calculation.
				For a local holomorphic coordinate $z=x+\sqrt{-1}y$ on an open subset $W\subset\Sigma_0$,
				there exists a positive function $\lambda$ on $W$ such that $g_{\Sigma} = \lambda(dx^2+dy^2)$.
				Then the fundamental form $\omega_M$ of $(M,J,g_M)$ can be written as $\omega=d\theta\wedge\alpha + \lambda dx\wedge dy$.
				Hence we have $\deebar\partial\omega_M =0$, and hence $(M,J,g_M)$ is a Gauduchon manifold.
			\end{proof}
			Let $(V,\deebar_V)$ be a mini-holomorphic bundle on $X$.
			The pullback $\tilde{V}:=p^{\ast}V$ has a natural holomorphic structure $\deebar_{\tilde{V}}$ determined as follows:
			\[
				\deebar_{\tilde{V}}(p^{\ast}s) = (p^{\ast}\deebar_V(s))^{0,1},
			\]
			where $s$ is a local section of $V$.
			We call $(\tilde{V},\deebar_{\tilde{V}})$ the flat lift of the mini-holomorphic bundle $(V,\deebar_V)$.
			For a Hermitian metric $h$ on $V$,
			the upstairs connection $p^{\ast}A_h + d\theta\otimes p^{\ast}\Phi_h$ is the Chern connection of $p^{\ast}h$.
			Moreover,
			$h$ is a BHE-metric if and only if $p^{\ast}h$ is a HE-metric.
			
			Let $Z\subset X$ be a finite subset.
			As a preparation to prove Theorem \ref{Thm:ex. of BHE-met.},
			we prove that $M'=M\setminus(S^1\times Z)$ satisfies the following assumptions in \cite{Ref:Zhen-Zhen-Zhen}.
			\begin{enumerate}[label=(\Roman*)]
				\item
					The volume of $M'$ is finite.
				\item
					There exists an exhaustion function $f$ of $M'$ such that $|\tilde{\Delta}(f)|<M$ for some $M>0$.
				\item
					There exist $C>0$ and an increasing function $a:[0,\infty)\to[0,\infty)$ with $a(0)=0$ and $a(x)=x$ for $x>1$,
					such that if for a bounded positive function $f$ on $X$ satisfies $\tilde{\Delta}(f)\leq B$ then $\sup{f} \leq C(B)a(||f||_{L^1})$.
					Furthermore,
					if $\tilde{\Delta}(f)\leq 0$, then $\tilde{\Delta}(f)=0$.
			\end{enumerate}
			\begin{Prp}\label{Prp:Assump in Zhen^3}
				If $X$  is compact,
				then the assumptions (I)-(III) in \cite{Ref:Zhen-Zhen-Zhen} holds for $M':=M\setminus(S^1\times Z)$.
			\end{Prp}
			\begin{proof}
				Obviously $M'$ has finite volume, and hence the assumption (I) holds for $M'$.
				We prove the assumption (II).
				By a direct calculation,
				there exists a vector field $\beta$ on $X$ such that we have $\tilde{\Delta}(p^{\ast}f) = 2p^{\ast}(\Delta(f) + \beta f)$ for any $f\in C^{\infty}(X)$.
				For each $p\in Z$,
				there exists a smooth function $\rho_p: B_{2\eps}(p)\setminus\{p\} \to \mathbb{R}$ that satisfies
				\[
				\left\{
					\begin{array}{l}
						(\Delta+\beta)(\rho_p)=0\\
						\rho_p(x) = d(x,p)^{-1} + O(1).
					\end{array}
				\right.
				\]
				Thus,
				by using a partition of unity,
				we obtain a non-negative smooth function $\rho:X\setminus Z \to \mathbb{R}$ satisfying
				\[
					\left\{
						\begin{array}{ll}
							|(\Delta+\beta)(\rho)|<R & (\exists R>0)\\
							\rho(x) = d(x,p)^{-1} + O(1) & (x\in B_{\eps}(p)).
						\end{array}
					\right.
				\]
				Therefore the pullback $\tilde{\rho}:=p^{\ast}\rho$ is an exhaustion function of $M'$ and $|\tilde{\Delta}(\tilde{\rho})|<R$,
				and this is the assertion of the assumption (II).
				We prove the assumption 3.
				Let $O(Z)$ be the orbit of $Z$ with respect to $\partial_t$-action.
				Since $M$ is compact,
				$S^1\times O(Z)$ is a smooth hypersurface of $M$.
				The assumption (III) holds for $M\setminus (S^1\times O(Z))$ by the same argument in \cite[Proposition 2.2]{Ref:Sim}.
				Therefore $M'$ also satisfies the assumption (III) because of the inclusions $M\setminus (S^1\times O(Z))\subset M'\subset M$.
			\end{proof}
		\subsubsection{The stability condition for mini-hol. bundles on mini-hol. manifolds}
			Let $X$ be a compact connected $3$-fold with a mini-holomorphic structure $(\partial_t,\alpha,\Sigma,\pi)$ and $Z\subset X$ a finite subset.
			Let $g_{\Sigma}$ be a K\"{a}hler metric on $\Sigma$ and set $g_{X}:= \alpha^2 + \pi^{\ast}g_{\Sigma}$.
			We set the orientation of $X$ as $\mathrm{vol}_X = \alpha\wedge\pi^{\ast}\mathrm{vol}_{\Sigma}$.
			Let $U_q\subset X$ be a sufficiently small neighborhood of $q\in Z$.
			Let $U_{q,4}\subset \mathbb{C}^2$ be a neighborhood of $(0,0)\in\mathbb{C}^2$ and $p_{q}:U_{q,4}\to U_{q}$ the Hopf-fibration by identifying $U_q$ with a neighborhood of $0\in\mathbb{R}^3$.
			We set the holomorphic structure on $U_{q,4}$ by Corollary \ref{Cor:hol. coord. on Hopf lift}.
			Let $(V,\deebar_V)$ be a Dirac-type singular mini-holomorphic bundle of rank $r>0$ on $(X,Z)$ such that each $q\in Z$ is of weight $\vec{k}_q=(k_{q,i})\in \mathbb{Z}^r$.
			Let $V_{q,4}$ be the holomorphic bundle on $U_{q,4}$ obtained by applying Proposition \ref{Prp:the lift of mini-hol. bundle} to $(V,\deebar_V)|_{U_q}$.
			\begin{Def}\label{Def:adm. met. on mini-hol.}
				\begin{enumerate}[label=(\roman*)]
					\item
						A smooth Hermitian metric $h$ on $(V,\deebar_V)$ is admissible if for any $q\in Z$ the pullback metric $p_{q}^{\ast}h$ can be prolonged to a Hermitian metric of $V_{q,4}$ of $C^1$-class.
					\item
						We define the degree $\mathrm{deg}(V,\deebar_V)$ to be
						\[
							\mathrm{deg}(V,\deebar_V):=
								\int_{X\setminus Z} \alpha\wedge c_1(A_h),
						\]
						where $h$ is an admissible Hermitian metric on $(V,\deebar_V)$.
					\item 
						We define the slope of $(V,\deebar_V)$ to be $\mu(V,\deebar_V):= \mathrm{deg}(V,\deebar_V)/\mathrm{rank}(V)$.
					\item 
						A mini-holomorphic bundle $(V,\deebar_{V})$ is said to be stable if $\mu(F,\deebar_F)< \mu(V,\deebar_V)$ holds for any proper mini-holomorphic subbundle $(F,\deebar_F)$ of $(V,\deebar_V)$.
						Semistability and polystability of mini-holomorphic bundles are also defined as in the usual case.
				\end{enumerate}
			\end{Def}
			\begin{Remark}
				By Lemma \ref{Lem:rest. of mini-hol. is regular},
				the restriction of an admissible metric $h$ to a mini-holomorphic subbundle $(F,\deebar_F)$ of $(V,\deebar_V)$ is also admissible.
			\end{Remark}
			We show some properties of admissible Hermitian metrics and well-definedness of the degree $\mathrm{deg}(V,\deebar_V)$.
			\begin{Prp}\label{Prp:char. of adm.}
				Any admissible Hermitian metrics on $(V,\deebar_V)$ are mutually bounded.
				Conversely,
				a Hermitian metric $\tilde{h}$ is admissible if the following conditions are satisfied:
				\begin{itemize}
					\item
						The metric $\tilde{h}$ and an admissible metric $h_0$ are mutually bounded. 
					\item
						For any $q\in Z$ there exists a neighborhood $U\subset X$ of $q$ such that the tuple $(V,h,A_h,\Phi_h)|_{U\setminus\{q\}}$ is a monopole on $U\setminus \{q\}$.
				\end{itemize}
			\end{Prp}
			\begin{proof}
				By the definition of the admissible metrics,
				the former claim is trivial.
				We prove the Converse.
				Let $p:U_4 \to U$ be the Hopf-fibration by identifying $U$ as a neighborhood of $0\in\mathbb{R}^3$.
				The pullback $p^{\ast}\tilde{h}$ is a Hermite-Einstein metric on $(V_4,\deebar_{V_4})|_{U\setminus\{q\}}$ and $p^{\ast}\tilde{h}$ and $p^{\ast}h_0$ are mutually bounded.
				Therefore $p^{\ast}\tilde{h}$ can be prolonged over $U_4$ by Lemma \ref{Lem:Removable singularity of HE-metric},
				and hence $(V,\tilde{h},A_{\tilde{h}},\Phi_{\tilde{h}})|_{U\setminus\{q\}}$ is a Dirac-type singular monopole on $(U,\{q\})$ by Theorem \ref{Thm:Pauly's criterion}.
			\end{proof}
			\begin{Prp}\label{Prp:well-definedness of degree}
				The degree $\mathrm{deg}(V,\deebar_V)$ is independent of the choice of admissible Hermitian metrics.
			\end{Prp}
			\begin{proof}
				Let $h_1$ and $h_2$ be admissible Hermitian metrics on $(V,\deebar_V)$ and $(A_i,\Phi_i)$ the CH connections and endomorphisms for $i=1,2$.
				Fix $q\in Z$ and take a neighborhood $U\subset X$ of $q$.
				As the proof of the last proposition,
				we take the Hopf fibration $p:U_4\to U$ and the holomorphic bundle $(V_4,\deebar_{V_4})$ on $U_4$.
				Then $p^{\ast}h_i$ are the Hermitian metrics on $V_4$ and the upstairs connections $A_{4,i} = p^{\ast}A_i + \xi\otimes p^{\ast}\Phi_i$ are the Chern connections of $p^{\ast}h_i$ respectively.
				Since $A_{4,i}$ are at least $C^0$-connections on $V_4$, we have $A_{4,1}- A_{4,2} = O(1)$,
				and hence we obtain an estimate
				\begin{align}
					|A_{1}- A_{2}|,|\Phi_1-\Phi_2| = O(\mathrm{dist}(\cdot,q)^{-1}).
				\end{align}
				Set $M:=S^1_{\theta}\times X$.
				Let $p_M:M\to X$ be the projection and $\omega_M$ the fundamental form of $M$.
				Set $B_{\eps}(Z):= \coprod_{p\in Z} B_{\eps}(p)$ for $\eps>0$.
				Then by using the flat lift $(\tilde{V},\deebar_{\tilde{V}})$ of $(V,\deebar_{V})$ we can wright
				\begin{align*}
					 &\int_{X\setminus Z} \alpha\wedge\left(\mathrm{Tr}(F(A_{h_1}))-\mathrm{Tr}(F(A_{h_2}))\right)\\
					=&(2\pi)^{-1}\int_{M\setminus (S^1\times Z)} \omega_M\wedge \mathrm{Tr}\left(F(A_{p_{M}^{\ast}h_1})-F(A_{p_{M}^{\ast}h_2})\right)\\
					=&(2\pi)^{-1}\lim_{\eps\to+0}\int_{M\setminus (S^1\times B_{\eps}(Z))} \omega_M\wedge \mathrm{Tr}\left(F(A_{p_{M}^{\ast}h_1})-F(A_{p_{M}^{\ast}h_2})\right).
				\end{align*}
				Set $\eta := \mathrm{det}(h_1h_2^{-1})$.
				Then we have
				\begin{align*}
					&(2\pi)^{-1}\int_{M\setminus (S^1\times B_{\eps}(Z))} \omega_M\wedge \mathrm{Tr}\left(F(A_{p_{M}^{\ast}h_1})-F(A_{p_{M}^{\ast}h_2})\right)\\
					=&(2\pi)^{-1}\int_{M\setminus (S^1\times B_{\eps}(Z))} \omega_M\wedge \deebar\partial p_{M}^{\ast}\eta\\
					=&(2\pi)^{-1}\int_{M\setminus (S^1\times B_{\eps}(Z))} d\{\partial(p_{M}^{\ast}\eta\cdot\omega_M)-p_{M}^{\ast}\eta(\partial\omega_M-\cdot\deebar\omega_M)\}\\
					=&(2\pi)^{-1}\int_{S^1\times\partial B_{\eps}(Z)}\partial(p_{M}^{\ast}\eta\cdot\omega_M)-p_{M}^{\ast}\eta(\partial\omega_M-\deebar\omega_M)\\
					=&(2\pi)^{-1}\int_{S^1\times\partial B_{\eps}(Z)} \{\mathrm{Tr}(A_{p_{M}^{\ast}h_1}-A_{p_{M}^{\ast}h_2})\wedge\omega_M + p_{M}^{\ast}\eta\cdot\partial\omega_M\} + O(\eps^2)\\
					=&\;O(\eps)\;\;(\because (1)).
				\end{align*}
				Taking the limit $\eps\to+0$ we obtain
				\[
					\int_{X\setminus Z} \alpha\wedge\mathrm{Tr}(F(A_{h_1})) =\int_{X\setminus Z}\alpha\wedge\mathrm{Tr}(F(A_{h_2})),
				\]
				which proves the uniqueness.
			\end{proof}
			Polystability of the underlying mini-holomorphic bundle of a Dirac-type HE-monopole easily follows from the Gauss-Codazzi formula as in the ordinary Kobayashi-Hitchin correspondence.
			We prove the converse.
			\begin{Thm}\label{Thm:ex. of BHE-met.}
				If $(V,\deebar_V)$ is stable,
				then there exists an admissible BHE-metric $h$ on $(V,\deebar_V)$.
			\end{Thm}
			\begin{proof}
				We take an admissible metric $h_0$ on $(V,\deebar_V)$.
				For a Dirac type singular mini-holomorphic bundle $(E,\deebar_E)$ on $(X,Z)$ and an admissible Hermitian metric $h_E$ on $E$,
				we have 
				\[
					\mathrm{deg}(E,\deebar_E)=
					(2\pi)^{-1}\int_{M\setminus(S^1\times Z)} \omega_M\wedge c_1(A_{p^{\ast}h_E}) = (2\pi)^{-1}\mathrm{deg}(\tilde{E},p^{\ast}h_E),
				\]
				where $(\tilde{E},\deebar_{\tilde{E}})$ is the flat lift of $(E,\deebar_E)$.
				Therefore the slope inequality $\mu(\tilde{V},p^{\ast}h_0) > \mu(\tilde{F},p^{\ast}(h_0|_{F}))$ holds for any proper mini-holomorphic subbundle $F$ of $V$.
				Since an $S^1$-invariant saturated subsheaf $\mathcal{F}$ of $\tilde{V}$ is an $S^1$-invariant holomorphic subbundle,
				there exists a mini-holomorphic subbundle $(F,\deebar_F)$ such that it satisfies $\mathcal{F}=\tilde{F}$,
				where $(\tilde{F},\deebar_{\tilde{F}})$ is the flat lift of $(F,\deebar_F)$.
				Hence the slope inequality $\mu(\tilde{V},p^{\ast}h_0) > \mu(\mathcal{F},p^{\ast}h_0|_{\mathcal{F}})$ also holds for any proper saturated $S^1$-invariant subsheaf $\mathcal{F}$.
				By \cite[Theorem 1.1]{Ref:Zhen-Zhen-Zhen} and Proposition \ref{Prp:Assump in Zhen^3},
				there exists a HE-metric $\tilde{h}$ of $\tilde{V}$ such that $\tilde{h}$ and $p^{\ast}h_0$ is mutually bounded.
				Let $h$ be the descent of $\tilde{h}$.
				Then $h$ is BHE-metric and mutually bounded to $h_0$.
				Therefore $h$ is an admissible BHE-metric by Proposition \ref{Prp:char. of adm.},
				which proves the theorem.
			\end{proof}
			\begin{Remark}
				Indeed group actions are not considered in \cite[Theorem 1.1]{Ref:Zhen-Zhen-Zhen},
				however the proof of Theorem 1.1 remains valid for the case that $V$ has an action by a group $G$ and satisfies the slope inequality for any $G$-invariant saturated subsheaves.
			\end{Remark}